\documentclass[a4paper]{article}

\usepackage{textcomp}
\usepackage[utf8]{inputenc}
\usepackage[T1]{fontenc}

\usepackage{csquotes}
\usepackage[
    backend=biber,
    doi=false,
    eprint=false,
    giveninits=true,
    isbn=true,
    style=numeric,
    url=false,
    date = year,
]{biblatex}
\AtEveryBibitem{%
  \clearfield{issn}%
}

\usepackage{hyperref}
\usepackage[dvipsnames]{xcolor}
\hypersetup{
    colorlinks,
    linkcolor={red!60!black},
    citecolor={green!40!black}
}

\usepackage[caption=false]{subfig}
\usepackage{wrapfig}
\usepackage{amssymb}
\usepackage{amsmath}
\usepackage{amsthm}
\usepackage{mathtools}
\usepackage{algorithm,algorithmic}
\usepackage{stmaryrd}
\usepackage{booktabs}
\usepackage{nicefrac}
\usepackage{multirow}
\usepackage{enumitem}
\usepackage{graphicx}

\graphicspath{{../overleaf/}}

\definecolor{labelkey}{rgb}{0,0.08,0.45}
\definecolor{refkey}{rgb}{0,0.6,0.0}
\definecolor{Brown}{rgb}{0.45,0.0,0.05}
\definecolor{dgreen}{rgb}{0.00,0.49,0.00}
\definecolor{dblue}{rgb}{0,0.08,0.75}
\definecolor{ffwwqq}{rgb}{1.,0.4,0.}
\definecolor{qqzzqq}{rgb}{0.,0.6,0.}
\definecolor{qqqqff}{rgb}{0.,0.,1.}
\definecolor{dred}{HTML}{D90404}
\definecolor{orng}{HTML}{D35400}
\definecolor{cb-black}      {RGB}{  0,   0,   0}
\definecolor{cb-blue-green} {RGB}{  0,  073,  073}
\definecolor{cb-green-sea}  {RGB}{  0, 146, 146}
\definecolor{cb-rose}       {RGB}{255, 109, 182}
\definecolor{cb-salmon-pink}{RGB}{255, 182, 119}
\definecolor{cb-purple}     {RGB}{ 73,   0, 146}
\definecolor{cb-blue}       {RGB}{ 0, 109, 219}
\definecolor{cb-lilac}      {RGB}{182, 109, 255}
\definecolor{cb-blue-sky}   {RGB}{109, 182, 255}
\definecolor{cb-blue-light} {RGB}{182, 219, 255}
\definecolor{cb-burgundy}   {RGB}{146,   0,   0}
\definecolor{cb-brown}      {RGB}{146,  73,   0}
\definecolor{cb-clay}       {RGB}{219, 209,   0}
\definecolor{cb-green-lime} {RGB}{ 36, 255,  36}
\definecolor{cb-yellow}     {RGB}{255, 255, 109}
\definecolor{bred}{HTML}{FF0000}
\definecolor{bpurp}{HTML}{BF00BF}
\definecolor{bblu}{HTML}{0000FF}
\definecolor{bcyan}{HTML}{00BFBF}
\definecolor{byellow}{HTML}{BFBF00}
\definecolor{bgreen}{HTML}{008000}

\usepackage{todonotes}

\newif\ifTODO 
\TODOtrue

\definecolor{darkgreen}{rgb}{0,0.6,0}

\newtheorem{theorem}{Theorem}[section]
\newtheorem{corollary}[theorem]{Corollary}

\newtheorem{lemma}[theorem]{Lemma}
\newtheorem{example}[theorem]{Example}

\newtheorem{definition}[theorem]{Definition}

\newtheorem{remark}[theorem]{Remark}

\newcommand{\cc}{\mathcal{C}}
\newcommand{\cO}{\mathcal{O}}
\newcommand{\rr}{\mathbb{R}}
\newcommand{\R}{\mathbb{R}}
\newcommand{\nn}{\mathbb{N}}

\newcommand*{\vsepfbox}[1]{%
  \begingroup
    \sbox0{\fbox{#1}}%
    \setlength{\fboxrule}{0pt}%
    \mbox{\kern-\fboxsep\fbox{\unhbox0}\kern-\fboxsep}%
  \endgroup
}
\newcommand{\abs}{\operatorname{abs}}

\newcommand{\cx}{\mathring{x}}
\newcommand{\cz}{\mathring{z}}

\definecolor{darkgreen}{rgb}{0,0.6,0}

\definecolor{darkred}{rgb}{0.6,0.0}

\definecolor{darkblue}{rgb}{0,0,0.6}

\definecolor{lightgraytext}{rgb}{0.6,0.6,0.6}

\DeclareMathOperator*{\argmin}{\mathrm{arg\,min}}

\title{Abs-Smooth Frank-Wolfe Method: Primal-Dual Analysis, Heavy Ball Momentum, and Inexact Oracles}

\author{Sri Harshitha Tadinada\thanks{Institut für Mathematik, Humboldt-Universität zu Berlin, Germany}\and Sebastian Pokutta\thanks{Zuse Institute Berlin, Germany}$\;^,$\thanks{Institut für Mathematik, Technische Universität Berlin, Germany} \and Andrea Walther$^*$$^,$$^\dagger$}

\date{\today}

\begin{document}

\maketitle

\begin{abstract}
We study projection-free optimization for convex objectives that satisfy abs-smoothness,
a structural property that
captures many non-smooth yet piecewise smooth functions arising, e.g., in modern
machine learning models. 
We develop a unified framework for Abs-Smooth Frank-Wolfe methods,
establishing a clean primal-dual analysis
that guarantees convergence without requiring classical smoothness assumptions. 
Our framework extends the available results in two important directions.
First, we introduce a heavy ball momentum variant and show that 
momentum can be incorporated naturally under abs-smoothness while preserving convergence guarantees. 
Second, we analyze inexact minimization oracles, demonstrating robustness to 
approximate inner solutions.
Moreover, we relax the full convexity assumption and study the case where
convexity holds only for the piecewise linear approximations of the objective,
further broadening the applicability of conditional gradient methods to a wider 
class of non-smooth problems.
\end{abstract}

\bigskip
\noindent

\section{Introduction} 
Constrained optimization problems arise in many applications. 
One way to solve such problems is provided by so-called \emph{projection-free} methods, in particular the \emph{conditional gradient method} \cite{CG66} (also known as \emph{Frank-Wolfe algorithm} \cite{FW56}).
This broad family of methods has seen many applications in the smooth setting \cite{Ga16, FWApp3, FWApp2, MBS21, TTP21, FWApp1} and in fact it is mostly formulated for smooth problems with very few exceptions; we refer the interested reader to \cite{CGFWSurvey2022} for a comprehensive overview.

Here, we are interested in solving non-smooth constrained optimization problems.
More specifically, we consider minimizing an abs-smooth \cite{Gr13} function $f\colon\rr^n\to\rr$
subject to a compact and convex constraint set $C\subseteq\rr^n$, i.e., 
\begin{equation}
\label{eq:p}
\underset{x\in C}{\text{min}}\;\; f(x).
\end{equation}
The class of abs-smooth functions can be briefly described as the class generated by smooth 
functions, $\min$, $\max$, and their compositions.
Essentially all functions whose non-differentiability arises due to the absolute-value function
fall into this large class,
including many widely used functions in machine learning,
such as hinge loss from support vector machines (SVMs) and neural network losses with ReLU activations.

For convex, non-smooth optimization problems, classical 
first-order methods achieve sublinear convergence rates in
terms of function values.
In particular, the subgradient method guarantees that the
error $f(x_t)-f(x^*)$ decreases at a rate of 
$\cO(1/\sqrt{t})$ for general convex Lipschitz functions,
where $t$ denotes the iteration number \cite{Nesterov04, Shor85}.
When additional structure such as strong convexity is present,
this rate can be improved to $\cO(1/t)$ \cite{Nesterov04}.
A variety of alternative approaches have been developed to
enhance both theoretical and practical performance,
including bundle methods \cite{Kiwiel95}, 
which aggregate subgradient information over multiple iterations,
and conditional gradient variants tailored to 
non-smooth objectives \cite{Harchaoui15},
which leverage linearization and domain structure.
These methods typically achieve similar sublinear 
convergence guarantees while offering improved robustness
or computational efficiency in large-scale settings.
Moreover, accelerated schemes for non-smooth convex 
optimization, such as those proposed by 
Nesterov~\cite{Nesterov05}, further improve iteration
complexity in terms of dependence on the desired accuracy,
even though the asymptotic rate of $\cO(1/t)$ remains 
characteristic for general non-smooth problems.
Collectively, these results highlight that while
non-smooth convex optimization is intrinsically more
challenging than its smooth counterpart, principled 
algorithmic design can still ensure provably efficient convergence.

An important class of non-smooth but still highly structured problems are problems with so-called abs-smooth objective functions, basically compositions of smooth functions with absolute value functions; we make this precise further below. In \cite{KrPoWaWo23} the \emph{Abs-Smooth Frank-Wolfe algorithm (ASFW)} was introduced and
initial convergence guarantees in the broader context
of abs-smooth, non-convex optimization were established. 
Building upon this foundation, the present paper focuses
on the convex setting and provides a refined theoretical
analysis tailored to convex abs-smooth functions.
We derive explicit convergence rates and show that, 
under suitable assumptions, the proposed method achieves
rates comparable to those known for smooth convex optimization,
as established in classical results for conditional 
gradient methods \cite{FW56,Ja13,CG66}. 
Moreover, we relax the assumption of full convexity of $f$
and also study the case where convexity is required only for
its piecewise linear approximation, 
allowing the analysis to apply to a broader class of functions; albeit with significantly weaker convergence guarantees.
These findings demonstrate that the abs-smooth structure
enables a favorable balance between modeling flexibility
and algorithmic efficiency, thereby bridging the gap between
smooth and non-smooth optimization frameworks.

To accomplish this, we show that the primal-dual gap for convex
abs-smooth functions with agnostic step-sizes of order $\cO(1/t)$ is itself bounded 
by $\cO(1/t)$.
We prove this convergence rate for two variants of the Abs-Smooth Frank-Wolfe 
algorithm (ASFW), namely, the \textit{vanilla ASFW} and 
\textit{heavy ball momentum ASFW} 
or heavy ball ASFW for short.
These variants have been appropriately adapted from their respective smooth 
versions of the algorithms \cite{Ja13,HB-fw}, by replacing gradients with
the \emph{abs-linearization} of a given abs-smooth function $f$.
Furthermore, we relax the inner minimization sub-problem of the vanilla 
ASFW \cite{KrPoWaWo23} to accommodate inexact solutions.
The implications of relaxing the inner sub-problem for a polyhedral domain
of interest $C$ are illustrated in later sections of this paper through a 
series of numerical examples, which demonstrate the practical impact of
this relaxation on solution quality, convergence behavior, and computational performance.

This article is organized as follows.
First, we discuss details and background of the considered function class. 
Then, we provide a brief overview of the ASFW framework and introduce
the different algorithmic variants considered in this study,
including the vanilla, relaxed vanilla, and heavy ball ASFW methods.
Subsequently, we present a primal-dual convergence analysis for each variant, highlighting 
the theoretical guarantees and key differences among the algorithms,
including the case where convexity is assumed only for the piecewise 
linear approximations of the objective $f$.
Finally, we corroborate the theoretical results through a collection of standard
non-smooth numerical test problems , 
demonstrating the practical performance and effectiveness of the proposed methods. 

\section{Abs-Smooth functions}
Several variants of the abs-smooth form were developed over time.
As demonstrated in \cite{ShKrBaBeWa23}, these variants are equivalent with respect
to important properties. 
In this work, we use the version that aligns most naturally with our formulation.
We define the abs-smooth class on the whole $\rr^n$ for simplicity.
For the theory developed in this paper, it would suffice that $f$ is
defined on an open set $\cal D$ with our domain of interest $C \subset \cal D$.

\begin{definition}[$\cc_{\abs}^d(\rr^n)$, abs-smooth functions]
\label{def:abssmooth}
For any $d \in \nn$, the set of locally Lipschitz continuous
functions $f: \rr^n \to \rr$, $y=f(x)$, that admit an
\emph{abs-smooth form}
\begin{align}
  \begin{split}
z_i & = F_i(x,z_1,\ldots,z_{i-1},|z_1|,\ldots,|z_{i-1}|),
\qquad i=1,\ldots,s, \\
y & = \varphi(x,z),
\end{split} \label{eq:abs}
\end{align}
for some $s \in \nn \cup \{0\}$, with $F=(F_1,\ldots,F_s)\in \cc^d(\rr^{n+s+s}, \rr^s)$
and $\varphi \in \cc^d(\rr^{n+s},\rr)$, is denoted by $\cc^d_{\abs}(\rr^n)$.
Thus the vector $z=(z_1,\ldots,z_s)$ is defined recursively.
For any $d\ge 1$, a function $f\in \cc^d_{\abs}(\rr^n)$ is called \emph{abs-smooth}.
The components $z_i$, $1 \leq i \leq s$, of $z$ are called \emph{switching variables}.
\end{definition}
The class of abs-smooth functions  first introduced  in \cite{Gr13} is quite broad. 
Therein, a function $f \in \cc^1(\rr^{n+s},\rr)$ is called \emph{composite piecewise differentiable}. 
In the terminology of Khan and Barton \cite{KhBa13,KhBa12}, the concept composite piecewise differentiability
is called abs-factorable. In the notation of Scholtes and others abs-smooth functions belong to the class
$PC^{1,1}$. 
Furthermore, the abs-smooth functions also represent a very small subset of the
class of lexicographically differentiable functions as
introduced by Nesterov \cite{Ne05}, which includes amongst others all convex functions.

For $f\in\cc_{\abs}^d(\rr^n)$, we consider the \emph{piecewise linearization} of $f$ 
localized at $\cx$ for any $ x \in \rr^n$ given by
\begin{equation}\label{eq:pl_model}
\begin{split}
    f_{PL,\cx}(x) &= d + a^T x + b^Tz,\\
    z &= c + Z x + Mz + L|z|,
\end{split}
\end{equation}
with the matrices and vectors 
\[
\begin{aligned}
Z &= \left.\frac{\partial}{\partial x} F(x, z, |z|)\right|_{(x,z,|z|)=(\cx, z(\cx), |z(\cx)|)}
    \in \mathbb{R}^{s \times n}, \\
M &= \left.\frac{\partial}{\partial z} F(x, z, |z|)\right|_{(x,z,|z|)=(\cx, z(\cx), |z(\cx)|)}
    \in \mathbb{R}^{s \times s} \quad \text{(strictly lower triangular)}, \\
L &= \left.\frac{\partial}{\partial |z|} F(x, z, |z|)\right|_{(x,z,|z|)=(\cx, z(\cx), |z(\cx)|)}
    \in \mathbb{R}^{s \times s} \quad \text{(strictly lower triangular)}, \\
a &= \left.\frac{\partial}{\partial x} \varphi(x, z)\right|_{(x,z)=(\cx, z(\cx))}
    \in \mathbb{R}^n, \;\;\;\;\;\;\;\;\;b = \left.\frac{\partial}{\partial z} \varphi(x, z)\right|_{(x,z)=(\cx, z(\cx))}
    \in \mathbb{R}^s.
\end{aligned}
\]
These derivatives are well defined due to the smoothness assumed for $F$ and $\varphi$. The constants $c\in \rr^s$ and $d\in \rr$ are chosen such that $f(\cx) = f_{PL,\cx}(\cx)$ holds \cite{KrPoWaWo23}.
Hence, $f_{PL,\cx}(\cdot)$ can be viewed as a piecewise linear model of $f$.
Moreover, $f_{PL,\cx}(\cdot)$ can be generated by suitably extended algorithmic differentiation (AD) \cite{Gr13};
see also \cite{Gr13,GW20,KrPoWaWo23} for details and further examples of abs-smooth functions.
The abs-linearization \cite{Gr13} of $f$ at $x$ is given by

\begin{equation}\label{eq:abs_pl}
  \Delta f(x; \cdot-x) = f_{PL,x}(\cdot) - f(x). 
\end{equation}
Note that, due to the choice of constants in equation~(\ref{eq:pl_model}), one has

\begin{equation*}
    \Delta f(x; 0) = f_{PL,x}(x) - f(x) = 0.
\end{equation*}
To illustrate abs-smooth functions and their piecewise linearizations, we next discuss a basic example.
\begin{figure}[t]
\centering
  \includegraphics[width=0.49\textwidth]{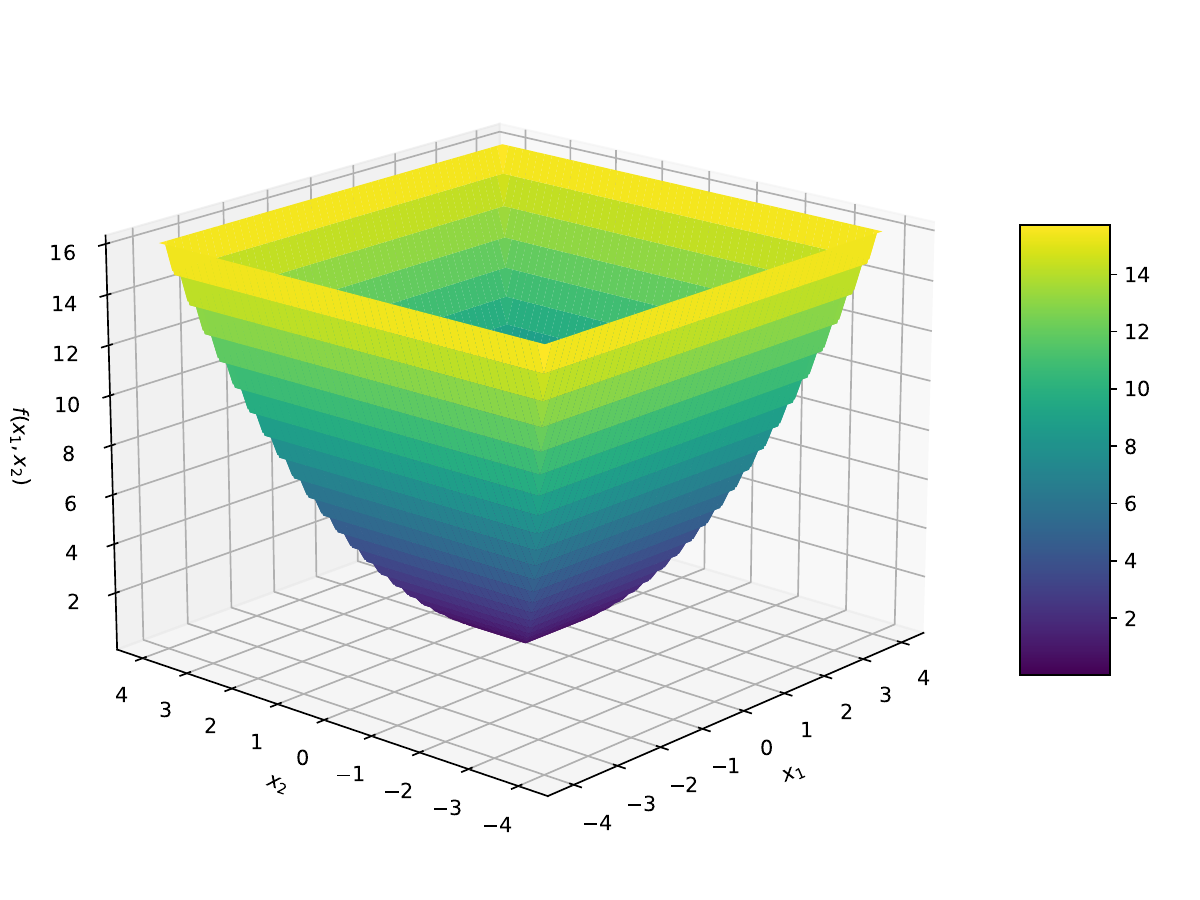}
  \hfill 
  \includegraphics[width=0.49\textwidth]{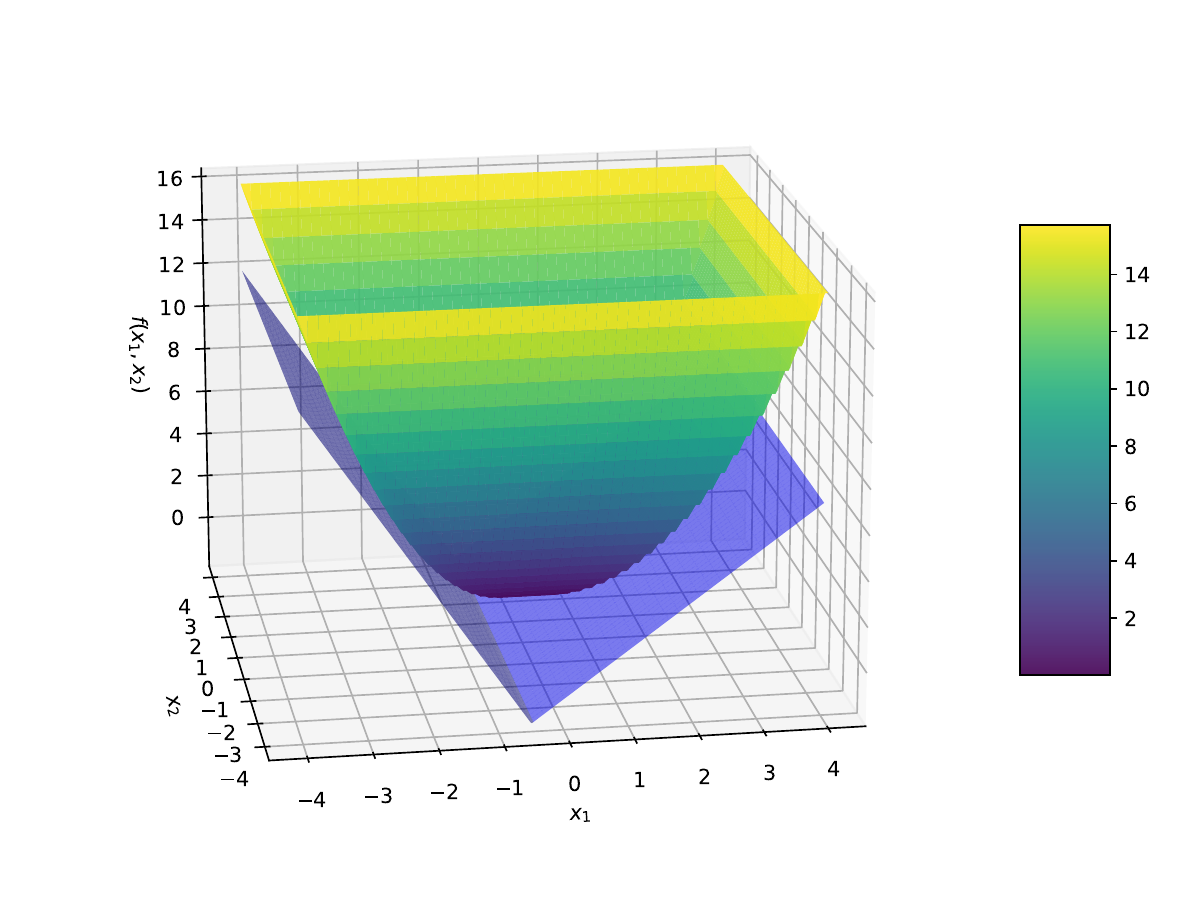}
\caption{Abs-smooth function from Example~\ref{ex} and its piecewise linear model~\eqref{pc_lin_model} in blue.}
\label{fig:example}
\end{figure}

\begin{example}[Simple example]\label{ex}
    The function $f:\rr^2 \rightarrow [0,\infty)$ given by $f(x_1,x_2) = \max(x_1^2,x_2^2)$ is abs-smooth since it can be written in the following form 
    \begin{equation*}
        f(x_1,x_2) = \max(x_1^2,x_2^2) = \frac{1}{2} \big( x_1^2 + x_2^2 + |x_1^2 - x_2^2|\big).
    \end{equation*}
Then one obtains the abs-smooth representation
    \begin{align*}
        & z_1 = x_1^2 - x_2^2,\\
        & z_2 = |z_1|,\\
        & y = 0.5 x_1^2 + 0.5 x_2^2 + 0.5 z_2.
    \end{align*}
    For this simple example, the abs-linearization is given by $\Delta f(\cx; \cdot): \rr^2 \rightarrow \rr$, with
    \begin{align}\label{pc_lin_model}
        \Delta f(\cx; \Delta x) & = \cx_1 \Delta x_1 + \cx_2 \Delta x_2
        + 0.5 (| \cz_1 + \Delta \cz_1| - |\cz_1|)\\
        & = \cx_1 \Delta x_1 + \cx_2 \Delta x_2 + 0.5 \big( |\cx_1^2 - \cx_2^2
        + 2\cx_1 \Delta x_1 - 2\cx_2 \Delta x_2| - |\cx_1^2 - \cx_2^2|\big).
    \end{align}
    The function itself together with its piecewise linear model $f_{PL,\cx}(\cdot)$ at the point $\cx = (-2,1)$ is depicted in Figure~\ref{fig:example}.
\end{example}

It is important to note that the abs-linearization of $f$ can be generated in different ways 
depending on the representation of $f$.
Consider for example the constant zero function, i.e., $f(x) = 0$ for all $x\in C$. The piecewise linearization for this 
representation of $f$ at zero is $f_{PL,0}(x)= 0$ for all $x\in C$.
Alternatively, one can re-write the constant zero function as $f(x) = e^x - |e^x|$ for all $x\in C$,
this function has one switching variable
and at zero its piecewise linearization is given by $f_{PL,0}(x)= 1+x - |1+x|$ for all $x\in C$. However, 
once the representation of $f$ is fixed, its abs-linearization is unique.
Hence, we assume throughout the paper that the representation of the abs-smooth function is fixed. 
This ensures that we work with the same abs-linearization across all our results.

The abs-smooth functions are guaranteed to possess directional derivatives.
Furthermore, within a neighborhood around the localization point $\cx$, the directional derivatives coincide with the piecewise linear model based on abs-linearization \cite{Gr13}. 
A primary characteristic of the piecewise linearization is that it approximates $f$ by explicitly taking its non-smoothness into account.
To establish stronger convergence results in the convex case, we utilize the following approximation property of the abs-linearization.

\begin{theorem}
\label{thm:approx}
Suppose $f\in\cc_{\abs}^1(\rr^n)$ is composite Lipschitz continuously differentiable on ${\cal D} \subseteq \rr^n$, ${\cal D}$ open, with $C \subset {\cal D}$ compact and convex. Then, there exists $\gamma > 0$ such that for all $x, \cx \in C$
\begin{align}
\label{eq:smooth}
|f(x)\!-\!f_{PL,\cx}(x)| = |f(x)\!-\!f(\cx)\!-\!\Delta f(\cx; x\!-\!\cx)| \leq \gamma \|x-\cx\|^2.
\end{align}
\end{theorem}

\begin{proof}
See \cite[Proposition~1]{Gr13}.
\end{proof}

Theorem~\ref{thm:approx} bounds the approximation error via the piecewise linearization.
It can also be viewed as the worst-case error that is being induced by the bounded feasible 
region $C$ on our approximate model.
Hence, one can interpret the piecewise linear model as a generalized Taylor
expansion at $\cx$ which accounts for non-smoothness and also maintains second-order accuracy.

\section{The Abs-Smooth Frank-Wolfe algorithm and its analysis}
First, we provide an overview of the Abs-Smooth Frank-Wolfe algorithm 
in its various formulations.
Subsequently, we proceed to prove all the lemmas required for the
convergence results.

Throughout the paper we consider agnostic step-sizes of the form 
$\alpha_t = 2/(t+2)$ for the sake of simplicity.
A similar analysis can also be done for step-sizes of the form
$\alpha_t = l/(t+l), \; l\in \mathbb{N}_{\geq 2}$, 
which is to be addressed in subsequent research. 

\subsection{Our Abs-Smooth Frank-Wolfe algorithm}
For $\epsilon > 0$, Algorithm~\ref{alg:asfw_hbasfw} yields a relaxed version of ASFW. 
In Step~\ref{a:s2}  and \ref{a:hb_s2}, respectively, of Algorithm~\ref{alg:asfw_hbasfw} we solve a piecewise 
linear sub-problem as opposed to a linear sub-problem in the smooth case. This sub-problem need not be solved exactly. Hence, 
we allow $v_t$ to
be an approximate solution of the piecewise linear sub-problem.
The constant $\cc_f$ represents the curvature constant associated with the 
abs-smooth function under consideration, which is formally defined in the next section. 
By setting $\epsilon=0$, Algorithm~\ref{alg:asfw_hbasfw} reduces to
the \emph{vanilla} ASFW algorithm proposed in \cite{KrPoWaWo23}.

A sketch of the
\emph{heavy ball} version of the ASFW method is also presented in
Algorithm~\ref{alg:asfw_hbasfw}.
We briefly explain the idea of the \emph{heavy ball} ASFW, as sketched in 
Algorithm~\ref{alg:asfw_hbasfw}. 
Similar to the Conditional Gradient momentum method in the smooth setting 
\cite{HB-fw}, which uses the weighted average of previous and current gradients 
to compute the next one, we give an appropriate adaptation when $f$ is abs-smooth.
The main idea of this algorithm is to use the data of previous abs-linearizations
in order to compute the next iterate.
Unlike the momentum based methods such as those in \cite{La-Ja15,NY18}, which 
require solving two sub-problems per iteration to evaluate the primal-dual error, 
our approach solves a single piecewise linear sub-problem at each iteration, 
while still attaining a comparable error bound. 
Even though computing and storing the previous abs-linearizations can be 
computationally demanding and memory-intensive, this variant offers meaningful 
theoretical insights.
Although it may not consistently outperform the vanilla algorithm in practice, 
we include a detailed sketch and analysis of the heavy ball momentum ASFW to 
highlight its potential and underlying principles.
The choice of weights $a_t$ from Algorithm \ref{alg:asfw_hbasfw}, which becomes 
clearer in the following sections, emphasizes more recent abs-linearizations by 
assigning them greater weight.

\begin{algorithm}[t]
\caption{Abs-Smooth Frank-Wolfe and Heavy ball ASFW algorithms}
\label{alg:asfw_hbasfw}  
\begin{algorithmic}[1]
\REQUIRE Point $x_0\in C$, abs-smooth function $f$, weights $a_t$,
$\cc_f >0$ and 
$\epsilon\geq 0$.
\FOR{$t=0$ \textbf{to} $\dotsc$}
\STATE Choose step-size $\alpha_t$  
\STATE \textcolor{gray}{$\diamond$ Choose either \ref{a:s2} (relaxed-ASFW) or \ref{a:hb_s2} (HB-ASFW) for the entire run}
\STATE Compute $v_t\in C$ such that $\Delta f(x_t;\alpha_t(v_t-x_t))
\leq \min_{w\in C}\Delta f(x_t;\alpha_t(w-x_t)) 
+ \frac{1}{2}\epsilon\alpha_t^2 \cc_f$
\label{a:s2}
\STATE Compute $v_t\in C$ such that $\sum\limits_{i=0}^{t} a_i \frac{\Delta f(x_i;\alpha_i(v_t-x_i))}{\alpha_i}
\leq \min_{v\in C} \sum\limits_{i=0}^{t} a_i \frac{\Delta f(x_i;\alpha_i(v-x_i))}{\alpha_i}
+ \epsilon a_t\alpha_t \cc_f$
\label{a:hb_s2}
\STATE $x_{t+1}= (1-\alpha_t)x_t+\alpha_tv_t$ 
\ENDFOR
\end{algorithmic}
\end{algorithm}

\begin{algorithm}[b]
\caption{Adapted ASM (AASM)}
\label{alg:aasm}  
\begin{algorithmic}[1]
\REQUIRE $\cx\in \mathbb{R}^n$, abs-linearization
$\Delta f(\cx,\alpha(\cdot-\cx))$
\STATE Initialize $P\leftarrow P_{\sigma(\cx)}$
\FOR{$r=0$ \textbf{to} $\dotsc$}
\STATE Compute $v_*\in\arg\min\limits_{v\in\overline{P}\cap C}\Delta f(\cx,\alpha(v-\cx))$
\label{aasm:2}
\IF{$v_*\in C$ local minimizer of $\Delta f(\cx,\alpha(\cdot-\cx))$} 
\STATE
Return $v_*$.
\ELSE
\STATE
$P \leftarrow P_{\sigma^+}$ which guarantees $P_{\sigma^+}\cap C\neq\varnothing$ and descent of $\Delta f(\cx,\alpha(\cdot-\cx))$
\label{aasm:pplus}
\ENDIF
\ENDFOR
\end{algorithmic}
\end{algorithm}

For the solution of the sub-problem in Step \ref{a:s2} and \ref{a:hb_s2}, respectively, of Algorithm~\ref{alg:asfw_hbasfw}, 
Algorithm~\ref{alg:aasm} provides a suitable adaptation of the
Active Signature Method (ASM) \cite{GW18} when the feasible set 
$C$ is polyhedral. If the abs-linear model contains $s$ 
absolute-value expressions, then 
for a fixed \emph{signature vector} $\sigma\in\{-1,0,1\}^s$, the inverse image 
$$ P_\sigma  \; = \;  \{  x \in \R^n : sgn(z( x)) = \sigma \}   $$
is called a {\em signature domain}.  These regions
$(P_\sigma)_{\sigma\in\{-1,0,1\}^s}$ are relatively open polyhedra
that form a partition of $\R^n$, see \cite{GW16} for more details.
On each such region,  the model reduces to a linear function.
Linear equality and inequality constraints describing 
$C$ can be incorporated directly into such a signature domain,
so that each sub-problem is defined over a single polyhedral
domain $P$.
Since $C$ is compact, the abs-linear model $\Delta f(\cx;\alpha(\cdot-\cx))$
admits a minimizer on every nonempty  polyhedral                
sub-domain, 
allowing the quadratic regularization required in the
unconstrained case \cite{GW18} to be omitted.
Consequently, the sub-problem in Line~\ref{aasm:2} of Algorithm~\ref{alg:aasm}
reduces
to a single linear program, possibly with up to $s$ 
additional linear inequalities compared to optimization over 
$C$ alone, where $s$ denotes the number of absolute-value
terms defining the piecewise linear structure.
The algorithm then follows the standard ASM logic,
checking the optimality conditions of \cite{GW18} with $Q=0$.
If optimality is not detected, it transitions to a new polyhedron 
$P_{\sigma^+}$ guaranteeing descent. 
In practice, this augmented ASM (AASM) typically
requires only a small number of LP solves, though 
in worst-case scenarios it may visit all signature
domains in a Klee-Minty-type exhaustive search \cite{GW18}.
For details, see \cite{KrPoWaWo23, paper1}.
Furthermore, note that Algorithm~\ref{alg:aasm} computes
a global minimizer if the abs-linearization is convex
but only a local one if this is not the case.
This flexibility is partially reflected in Algorithm~\ref{alg:asfw_hbasfw}
through the inexact-solve option.
Alternatively, for the relaxed variant of the sub-problem
in Step~\ref{a:s2} and \ref{a:hb_s2}, respectively, of Algorithm~\ref{alg:asfw_hbasfw},
we may restrict the number of
LP calls arising from the AASM algorithm.

\subsection{Convergence results}\label{sec:conv_results}
Now we introduce some useful definitions before stating the main results 
of this section.
\begin{definition}(Convex function)
    A function $f:C\subseteq\rr^n \rightarrow \rr$ is convex if for any $x,y \in C$ and any $\lambda \in [0,1]$ one has
    \begin{equation*}
        f(\lambda x + (1-\lambda) y) \leq \lambda f(x) + (1-\lambda) f(y).
    \end{equation*}
\end{definition}
\begin{definition}(Convex piecewise linearization)
Consider $f\in \cc^d_{\abs}(\rr^n)$ with its piecewise linearization at $x$ given by $f_{PL,x}(\cdot)$.
We say $f_{PL,x}(\cdot)$ is convex at $x$ if for any $y,z \in C$ and any $\lambda \in [0, 1]$
one has
\begin{equation*}
    f_{PL,x}(\lambda y + (1-\lambda)z) \leq \lambda f_{PL,x}(y) + (1-\lambda)f_{PL,x}(z).
\end{equation*}    
\end{definition}
Equivalently, one can rewrite the definition of convex piecewise linearization using the
abs-linearizations $\Delta f(x; \cdot -x)$ with the help of equation~(\ref{eq:abs_pl}) as follows
\begin{equation}\label{eq:cvx_abs}
    \Delta f (x; \lambda (y-z) + z-x) \leq \lambda \Delta f(x;(y-x)) + (1-\lambda) \Delta f(x;(z-x)),
\end{equation}
with $\lambda \in [0,1]$.
If either $y=x$ or $z=x$ in equation~(\ref{eq:cvx_abs}),
then for any $\alpha\in [0,1]$,
and with $x$ denoting the point of localization,
the convex abs-linearization satisfies
\begin{equation*}
    \Delta f (x; \alpha \Delta x) \leq \alpha \Delta f(x; \Delta x),
\end{equation*}
given any $\Delta x = \cdot - x \in \rr^n$.
We rely on this equivalent formulation of convex piecewise linearization 
from now on.
Furthermore, for a convex abs-smooth function $f$,
we just assume that $f$ is convex with a fixed representation as discussed above and nothing else is known about its piecewise linearization.
It is shown in \cite{WaGr19} that when the abs-smooth function $f$ is convex, then its piecewise linearization is locally convex
in a ball of some radius around the point of localization. 
A criterion to check local convexity of the generated piecewise linearization is also provided in \cite{WaGr19}.
In general, convexity of $f$ need not imply the global convexity of its piecewise linearization $f_{PL}$.

A key ingredient in the convergence
analysis of Frank-Wolfe type algorithms
is a measure of the \emph{nonlinearity} of the objective function $f$
over the domain $C$. 
The curvature constant $\cc_f$ of a convex and differentiable function 
$f:\rr^n\rightarrow \rr$, with respect to the convex compact 
domain $C$ is defined as
\begin{equation*}
        \cc_f := \sup_{\substack{x,v \in C, \\ \alpha \in (0,1], \\ y = x + \alpha (v-x)}} 
\frac{2}{\alpha^2} \Big( f(y) - f(x) - \langle y-x, \nabla f(x) \rangle \Big).
\end{equation*}
In the smooth setting \cite{Ja13}, the curvature constant measures the
deviation of $f$ at $y$ from the linearization of $f$ at $x$. 
In our setting, we modify $\cc_f$ so that it also captures
the deviation of $f$ at $y$ from the piecewise linear 
approximation of $f$ around $x$.
Since the piecewise linearization need not be a supporting model, i.e., an underestimator, we bound this deviation in absolute value.
\begin{definition}(Curvature constant for abs-smooth functions)
\label{def:curvature}
    Let $f\in \cc^d_{\abs}(\rr^n)$ be a function over the compact convex set $C\subseteq \rr^n$.
    We call a constant $\cc_f\in[0,\infty)$ a \emph{curvature bound} of $f$ on $C$ if for all
    $x,v \in C$ and $\alpha \in (0,1]$, with $y = x + \alpha (v-x)$, one has
    \begin{equation}
        \big|f(y) - f(x) - \Delta f(x; y-x)\big| \leq \frac{\alpha^2}{2}\cc_f.
    \end{equation}
    The smallest such bound is given by
    \begin{equation}
        \cc_f := \sup_{\substack{x,v \in C, \\ \alpha \in (0,1], \\ y = x + \alpha (v-x)}} 
\frac{2}{\alpha^2} \big| f(y) - f(x) - \Delta f(x; y-x) \big|.
    \end{equation}
\end{definition}
The curvature constant  $\cc_f$  for an abs-smooth function is determined by the choice of the representation of $f$.
For instance, if $f$ is piecewise linear and its piecewise linearization $f_{PL}$ is the function $f$ itself, 
then one would have $\cc_f=0$.
Theorem~\ref{thm:approx} implies that for any composite Lipschitz continuously differentiable function $f\in\cc_{\abs}^1(\rr^n)$, $\cc_f$ is finite on the compact set $C$ and satisfies $\cc_f \le 2\gamma D^2$, where $D$ is the diameter of $C$ and $\gamma$ is the constant in~(\ref{eq:smooth}).
The constant $\gamma$ can be bounded from above by the maximum of Lipschitz constants of the elemental functions
used to evaluate the $f$ \cite{Gr13}. 
Hence, $\cc_f$ is also bounded from above, which is an important assumption for our analysis.
In particular, one may use the explicit curvature bound $\cc_f := 2\gamma D^2$ in the statements below.
\begin{definition}\label{def:primal_gap}(Primal gap)
The primal gap of a function $f:C\subseteq\rr^n\rightarrow\rr$ at
point $x$ is given by 
\begin{equation}
  h(x)=\max\limits_{y\in C}\big(f(x) - f(y)\big) = f(x) - f(x^*)  
\end{equation}
with $x^*$ being a minimizer of $f$.
\end{definition}
In general, the knowledge of a minimizer $x^*$ or the optimal
function value $f(x^*)$ is required to compute the primal gap.
A usual remedy to this situation is to make use of the
\textit{Frank-Wolfe gap}, which can be computed without the
knowledge of $x^*$,
bringing us the definition of the Frank-Wolfe gap 
for abs-smooth functions, as first introduced in \cite{KrPoWaWo23}. 

\begin{definition}\label{def:fw_gap}(Abs-Smooth Frank-Wolfe gap)
The Frank-Wolfe gap for an abs-smooth function $f:C\subseteq\rr^n\rightarrow\rr$ and a step-size $\alpha\in(0,1]$ is
defined as
\begin{equation}
    g(x) = \max\limits_{v\in C}\frac{-\Delta f(x;\alpha(v-x))}{\alpha},
\end{equation}
with $\Delta f(x;(\cdot-x))$ being the abs-linearization of $f$ at $x$.
\end{definition}
In the smooth setting, we know that
$\Delta f(x;\cdot) = \langle \nabla f(x),\cdot\rangle$ \cite{Gr13}, 
and we recover the traditional Frank-Wolfe gap for smooth functions, which upper bounds the (standard) dual gap $\langle \nabla f(x), x - x^*\rangle$ for any optimal solution $x^*$; slightly abusing notions here we will refer to the Abs-smooth Frank-Wolfe gap also as \emph{dual gap}. 
For any iterate $x_t$ and step-size $\alpha_t$, the (exact) dual gap is
\[
    g(x_t) = \max_{v\in C}\frac{-\Delta f(x_t;\alpha_t(v-x_t))}{\alpha_t}.
\]
If the inner sub-problem in Step~\ref{a:s2} and \ref{a:hb_s2}, respectively, of Algorithm~\ref{alg:asfw_hbasfw} is solved exactly, i.e., $\epsilon=0$, and $v_t$ is a minimizer, then the dual gap can be evaluated as
\begin{equation}\label{eq:gx_t}
    g(x_t) = \frac{-\Delta f(x_t;\alpha_t(v_t-x_t))}{\alpha_t} = \frac{f(x_t)-f_{PL,x_t}(\tilde x)}{\alpha_t}\qquad\text{where}\;\tilde x = \alpha_tv_t-(\alpha_t-1)x_t,
\end{equation}
with the abs-linearization localized at $x_t$. It follows $g(x_t)\ge 0$ for all $t\ge 0$ since $v=x_t$ yields $\Delta f(x_t;0)=0$.
For inexact inner solves, we will refer to the computable quantity $\hat g_t := -\Delta f(x_t;\alpha_t(v_t-x_t))/\alpha_t$, which satisfies $\hat g_t\le g(x_t)$ and $\hat g_t\ge -\frac{\epsilon\alpha_t}{2}\cc_f$. If the inner sub-problem is solved exactly, then $\hat g_t = g_t$. From now on, 
we will use the abbreviations $g_t := g(x_t)$ with step-size $\alpha_t$ and $h_t := h(x_t)$.

It was shown in \cite{KrPoWaWo23} that for convex $f$, one has
\[
\min\limits_{v\in C} \Delta f(\cx; v-\cx) = 0 \implies f(\cx) = \min\limits_{x\in C}f(x).
\]
The other direction, i.e.,
\[
f(\cx) = \min\limits_{x\in C}f(x) \implies \min\limits_{v\in C} \Delta f(\cx; v-\cx) = 0.
\]
only holds true inside the ball where the piecewise linearization generated for the convex $f$
is also convex.
Hence, an optimality criterion is available, but it offers no insight
into its tightness, which is clarified in the following lemma.
\begin{lemma}(Dual bound)\label{l:lb}
For a convex $f\in \cc^d_{\abs}(\rr^n)$, let $x^*$ be a minimizer
of $f$ on the convex compact set $C$. Then, one has for any 
$x\in C$, $\alpha \in (0,1]$ and $\epsilon \geq 0$ arbitrary, that
\begin{equation*}
f(x) - f(x^*) \leq \frac{-\Delta f(x;\alpha(v-x))}{\alpha} 
+ \frac{\alpha}{2}  \cc_f (1+\epsilon),
\end{equation*}
where $v$ is an appropriate piecewise linear minimizer,
i.e.,
\begin{equation*}
    \Delta f(x;\alpha(v-x)) \leq \min_{w\in C}\Delta f(x;\alpha(w-x))
+ \frac{1}{2}\epsilon_\alpha \cc_f
\end{equation*}
with $\epsilon_\alpha = \epsilon \alpha^2$.
\end{lemma}

\begin{proof}
    Let $y = (1-\alpha)x + \alpha x^* \in C$. By convexity of $f$,
    i.e., $f(y) \le (1-\alpha)f(x) + \alpha f(x^*)$, which implies
    \begin{equation*}
        f(y) - f(x) \le \alpha (f(x^*)-f(x)).
    \end{equation*}
    and Definition~\ref{def:curvature} (curvature bound) applied to $y = x+\alpha(x^*-x)$, we obtain
    $$|f(y) - f(x) - \Delta f(x; \alpha(x^*-x))| \le \frac{\alpha^2}{2}\cc_f,$$ and hence
    \begin{equation*}
        f(y) - f(x) \ge \Delta f(x; \alpha(x^*-x)) - \frac{\alpha^2}{2}\cc_f.
    \end{equation*}
    Combining the two inequalities yields
    \begin{equation*}
        \alpha (f(x^*) - f(x)) \ge \Delta f(x; \alpha(x^*-x)) - \frac{\alpha^2}{2}\cc_f.
    \end{equation*}
    Dividing by $\alpha$ and rearranging gives
    \begin{equation*}
        f(x) - f(x^*) \le \frac{-\Delta f(x; \alpha(x^*-x))}{\alpha} + \frac{\alpha}{2}\cc_f.
    \end{equation*}
    Finally, since $v$ is an approximate minimizer of the inner problem, one has
    \begin{equation*}
        \min_{w\in C}\Delta f(x;\alpha(w-x)) \ge \Delta f(x;\alpha(v-x)) - \frac{1}{2}\epsilon_\alpha \cc_f,
    \end{equation*}
    and therefore
    \begin{equation*}
        g(x) = -\frac{1}{\alpha}\min_{w\in C}\Delta f(x;\alpha(w-x))
        \le \frac{-\Delta f(x;\alpha(v-x))}{\alpha} + \frac{\epsilon_\alpha}{2\alpha}\cc_f.
    \end{equation*}
    Using $\epsilon_\alpha = \epsilon\alpha^2$ and $-\Delta f(x;\alpha(x^*-x))/\alpha \le g(x)$ gives
    \begin{equation*}
        f(x) - f(x^*) \le \frac{-\Delta f(x;\alpha(v-x))}{\alpha} + \frac{\alpha}{2}\cc_f(1+\epsilon).
    \end{equation*}
\end{proof}
   
\begin{remark}\label{rmk:lb}
Let $f \in \cc^d_{\abs}(\rr^n)$ be convex on $C$ and let $x^*\in C$ be a minimizer of $f$. Then for any $x\in C$ and $\alpha\in(0,1]$,
\begin{equation}\label{eq:lb}
    f(x^*) \geq f(x) + \frac{\Delta f(x;\alpha(x^*-x))}{\alpha} - \frac{\alpha}{2} \cc_f.
\end{equation}
\end{remark}
\begin{lemma}(Dual bound for piecewise linearization)\label{l:dual_bound_pc_lin}
For an $f\in \cc^d_{\abs}(\rr^n)$ with convex piecewise linearization $f_{PL,x}(\cdot)$
at every $x\in C$,
let $x^*$ be a minimizer of $f$ on the convex compact set $C$.
Then, one has for any $x\in C$, $\alpha \in (0,1]$ and $\epsilon \geq 0$ arbitrary, that 
\begin{equation*}
    f(x) - f(x^*) \leq \frac{-\Delta f(x;\alpha(v-x))}{\alpha} 
+ \frac{1}{2} \big( 1 + \alpha\epsilon \big) \cc_f,
\end{equation*}
where $v$ is an appropriate piecewise linear minimizer,
i.e.,
\begin{equation*}
    \Delta f(x;\alpha(v-x)) \leq \min_{w\in C}\Delta f(x;\alpha(w-x))
+ \frac{1}{2}\epsilon_\alpha \cc_f
\end{equation*}
with $\epsilon_\alpha = \epsilon \alpha^2$.
\end{lemma}

\begin{proof}
    Let $w\in\argmin_{u\in C}\Delta f(x;\alpha(u-x))$.
    By optimality of $w$, it follows that
    \begin{align*}
        -\Delta f(x;\alpha(w-x)) \geq  -\Delta f(x;\alpha(x^*-x)).
    \end{align*}
    By convexity of the piecewise linearization (equivalently, convexity of $\Delta f(x;\cdot)$) and $\Delta f(x;0)=0$, one has
    \[
        \Delta f(x;\alpha(x^*-x)) \leq \alpha \Delta f(x;x^*-x),
    \]
    and hence
    \begin{align*}
        -\Delta f(x;\alpha(x^*-x))
        &\geq -\alpha \Delta f(x;x^*-x)
        = \alpha\big(f(x)-f_{PL,x}(x^*)\big).
    \end{align*}
    By Definition~\ref{def:curvature} with $v=x^*$ and $\alpha=1$, we have
    $\big|f(x^*)-f_{PL,x}(x^*)\big|\le \cc_f/2$, and thus
    $$f(x)-f_{PL,x}(x^*) \ge f(x)-f(x^*)-\cc_f/2.$$ Combining the estimates gives
    \begin{equation*}
        -\Delta f(x;\alpha(w-x)) \geq \alpha\Big(f(x)-f(x^*)-\frac{\cc_f}{2}\Big).
    \end{equation*}
    Finally, since $v$ is an approximate minimizer of the inner problem, one has
    \[
        \Delta f(x;\alpha(v-x)) \le \Delta f(x;\alpha(w-x)) + \frac{1}{2}\epsilon_\alpha \cc_f,
    \]
    i.e.,
    \[
        -\Delta f(x;\alpha(v-x)) \ge -\Delta f(x;\alpha(w-x)) - \frac{1}{2}\epsilon_\alpha \cc_f.
    \]
    Dividing by $\alpha$ and using $\epsilon_\alpha=\epsilon\alpha^2$ yields the claim.
\end{proof}

\begin{lemma}(Primal progress)\label{rmk:primal_prog}
    Let $f\in \cc^d_{\abs}(\rr^n)$ be a convex function over the compact
    convex set $C\subseteq \rr^n$, $\alpha_t \in (0,1]$ and $x_t \in C$. 
    Then for a step $x_{t+1} = x_t + \alpha_t(v_t-x_t)$, it holds that
    \begin{equation*}
        f(x_{t+1}) \leq f(x_t) -\alpha_t g(x_t) + \frac{\alpha_t^2}{2}\cc_f (1+\epsilon)\;,
    \end{equation*}
    if $v_t$ is an appropriate piecewise linear minimizer, i.e.,
    \begin{equation*}
        \Delta f(x_t;\alpha_t(v_t-x_t)) \leq \min_{w\in C}\Delta f(x_t;\alpha_t(w-x_t)) 
    + \frac{1}{2}\epsilon_t \cc_f
    \end{equation*}
    with $\epsilon_t = \epsilon \alpha_t^2$.
\end{lemma}
\begin{proof}
    Let $x = x_t$ and $y = x_{t+1} = x_t + \alpha_t (v_t-x_t)$.
    From the definition of our curvature constant $\cc_f$ for 
    the convex abs-smooth function $f$, we have
    \begin{equation*}
        f(y) = f(x_t + \alpha_t(v_t-x_t)) \leq f(x_t) + \Delta f(x_t; y-x_t) + \frac{\alpha_t^2}{2}\cc_f\;.
    \end{equation*}
    Now we use that $v_t$ is a good descent direction on the piecewise linearization of $f$ at $x_t$.
    Formally, the point $v_t$ satisfies
    \begin{align*}
        \Delta f(x_t;\alpha_t(v_t-x_t)) \leq \min_{w\in C}\Delta f(x_t;\alpha_t(w-x_t)) + \frac{1}{2}\epsilon_t \cc_f  = - \alpha_t g(x_t) + \frac{1}{2}\epsilon_t \cc_f.
    \end{align*}
    Here we have plugged in the definition (\ref{def:fw_gap}) of the dual gap $g(x_t)$. Therefore, from the
    definition of $\epsilon_t = \epsilon \alpha_t^2$ we obtain 
    \begin{equation*}
        f(x_{t+1}) \leq f(x_t) -\alpha_t g(x_t) + \frac{\alpha_t^2}{2}\cc_f (1+\epsilon),
    \end{equation*}
    proving the lemma.
\end{proof}
Equivalently, one can rewrite the primal progress of Lemma~\ref{rmk:primal_prog} in terms of the primal gap
at each iteration, i.e.,  one has
\begin{equation}\label{eq:pp_relaxed}
     h_{t+1}\leq h_t - \alpha_t g_t + \frac{\alpha_t^2}{2} \cc_f (1+\epsilon).
\end{equation}
This is analogous to the primal progress equation in the smooth case \cite{Ja13}. 
Since the ASFW with open loop step-sizes is not a descent method, we need a different technique
compared to the convergence proofs with a line search or short step step-size rules.
Therefore, we use a method called \emph{Approximate Duality Gap Technique} (ADGT) \cite{ADGT} to show our
next theorem, see also \cite{MaPo25}. 
The main intuition of this method is to create an upper estimate $G_t$ of the optimality gap
$f(x_{t+1}) - f(x^*)$, with $x_{t+1}$ as the output of Algorithm~\ref{alg:asfw_hbasfw} at time $t$.
This estimate corresponds to the difference between an upper bound on
$f(x_{t+1})$, $U_t\geq f(x_{t+1})$ and a lower bound on $f(x^*)$, $L_t\leq f(x^*)$, so that 
\begin{equation}\label{eq:gt}
    f(x_{t+1}) - f(x^*) \leq U_t-L_t = G_t.
\end{equation}
This estimate $G_t$ is referred to as the \emph{approximate duality gap}. 
In the ADGT framework, the main task is to derive a one-step inequality of the form
\begin{equation}\label{eq:adgt_step}
    A_tG_t - A_{t-1}G_{t-1} \leq E_t
\end{equation}
for a suitable error sequence $(E_t)_{t\geq 0}$.
One then telescopes~(\ref{eq:adgt_step}) and divides by $A_t$.
The simplest choice of an upper bound and lower bound are $f(x_{t+1})$ and $f(x^*)$, respectively.
By choosing the lower bound as $f(x^*)$, it is unclear how one can control the left-hand side of~(\ref{eq:adgt_step}) without the knowledge of $x^*$. 
Conversely, by convexity of $f$ and Lemma~\ref{l:lb}, one obtains a lower bound that allows
evaluating $A_tL_t$ without the knowledge of $x^*$.  
The power and generality of the ADGT is illustrated in \cite{ADGT} for various first-order methods.
The primal-dual convergence in the smooth case is analyzed in a similar manner in \cite{MaPo25}.

\begin{algorithm}[t]
\caption{Revised ASFW and HB-ASFW algorithms}
\label{alg:revised_asfw_hbasfw}  
\begin{algorithmic}[1]
\REQUIRE Point $x_0\in C$, abs-smooth function $f$, weights $\alpha_t,\, a_t$ and $\epsilon \ge 0$.
\FOR{$t=0$ \textbf{to} $\dotsc$}
\STATE $a_t \longleftarrow 2t+2,\, A_t \longleftarrow \sum\limits_{i=0}^t a_i = (t+1)(t+2)$ and $\alpha_t = a_t/A_t$  
\STATE \textcolor{gray}{$\diamond$ Choose either \ref{a:s2} (relaxed-ASFW) or \ref{a:hb_s2} (HB-ASFW) for the entire run}
\STATE \textbf{(relaxed-ASFW)} Choose $v_t\in C$ such that $\Delta f(x_t;\alpha_t(v_t-x_t)) \leq \min_{w\in C}\Delta f(x_t;\alpha_t(w-x_t)) + \frac{1}{2}\epsilon\alpha_t^2 \cc_f$
\STATE \textbf{(HB-ASFW)} Choose $v_t\in C$ such that $\sum\limits_{i=0}^{t} a_i \frac{\Delta f(x_i;\alpha_i(v_t-x_i))}{\alpha_i} \leq \min_{v\in C}\sum\limits_{i=0}^{t} a_i \frac{\Delta f(x_i;\alpha_i(v-x_i))}{\alpha_i} + \epsilon a_t\alpha_t \cc_f$
\STATE $x_{t+1} \longleftarrow \frac{A_{t-1}}{A_t}x_t+\frac{a_t}{A_t}v_t$ 
\ENDFOR
\end{algorithmic}
\end{algorithm}
\begin{lemma}(ADGT telescoping)\label{lem:adgt_telescoping}
Let $(A_t)_{t\geq 0}$ be a sequence of positive numbers and $(G_t)_{t\geq 0}$ satisfy
\[
    f(x_{t+1}) - f(x^*) \leq G_t
\]
for all $t\geq 0$.
If
\[
    A_tG_t - A_{t-1}G_{t-1} \leq E_t,\qquad t\geq 0,
\]
with the convention $A_{-1}G_{-1}:=0$, then
\[
    f(x_{t+1}) - f(x^*) \leq G_t \leq \frac{1}{A_t}\sum_{i=0}^t E_i.
\]
\end{lemma}

\begin{proof}
For any $t\geq 0$, summing the one-step inequality from $i=0$ to $t$ yields
\[
    A_tG_t
    = \sum_{i=0}^t \big(A_iG_i-A_{i-1}G_{i-1}\big)
    \leq \sum_{i=0}^t E_i.
\]
Dividing by $A_t$ gives the claim.
\end{proof}
For clarity, we present a more explicit version of Algorithm~\ref{alg:asfw_hbasfw}
as Algorithm~\ref{alg:revised_asfw_hbasfw}.

\begin{theorem}\label{thm:p-d}(Primal-dual convergence for relaxed vanilla ASFW)
     Let $f\in\cc_{\abs}^d(\rr^n)$ be a convex function over a compact convex set $C\subseteq\rr^n$ with 
     diameter $D$. For each $t\geq 0$, the iterates $x_t$ of the relaxed ASFW 
     Algorithm~\ref{alg:revised_asfw_hbasfw}
     satisfy 
 \begin{equation*}
   f(x_{t+1}) -f(x^*)\leq  G_t \leq \frac{4 \cc_f}{t+2}(1+\epsilon),
 \end{equation*}
where $x^*\in C$ is an optimal solution to problem (\ref{eq:p}),
$\cc_f$ the curvature constant of $f$, 
with $G_t = U_t-L_t$ for $U_t = f(x_{t+1})$ and
\begin{equation*}
   L_t =  \frac{1}{A_t} \bigg( \sum\limits_{i=0}^t a_if(x_i) +\sum\limits_{i=0}^t a_i\bigg[ \frac{\Delta f(x_i;\alpha_i(v_i-x_i))}{\alpha_i} - \frac{\alpha_i}{2} \cc_f(1+\epsilon)\bigg] \bigg).
\end{equation*}
\end{theorem}

\begin{proof}
Let $a_t \geq 0$ be numbers which will be determined later, and $A_t = \sum\limits_{i=0}^t a_i$. Set $A_{-1} = 0$.
The lower bound $L_t$ is obtained from the lower bound of $f(x^*)$ 
from Lemma \ref{l:lb}, 
\begin{equation*}
    f(x^*)  \geq f(x_t) + \frac{\Delta f(x_t;\alpha_t(v_t-x_t))}{\alpha_t} - \frac{\alpha_t}{2}\cc_f (1+\epsilon)
\end{equation*} 
and therefore also
\begin{align*}
     A_t f(x^*) & \geq \sum\limits_{i=0}^t a_if(x_i) +\sum\limits_{i=0}^t a_i\bigg[ \frac{\Delta f(x_i;\alpha_i(v_i-x_i))}{\alpha_i} - \frac{\alpha_i}{2} \cc_f(1+\epsilon)\bigg] =: A_t L_t.
\end{align*}
From this inequality, one obtains the lower bound $L_t$ as
\begin{equation}\label{eq:L_t_vanilla}
   L_t =  \frac{1}{A_t} \bigg( \sum\limits_{i=0}^t a_if(x_i) +\sum\limits_{i=0}^t a_i\bigg[ \frac{\Delta f(x_i;\alpha_i(v_i-x_i))}{\alpha_i} - \frac{\alpha_i}{2} \cc_f(1+\epsilon)\bigg] \bigg).
\end{equation}
Then, we define the gap $G_t := f(x_{t+1}) - L_t \geq 0$. The upper bound $f(x_{t+1})$ is one step ahead, which helps with the analysis.
For $t\geq 0$, one has
\begin{align*}
    A_tG_t - A_{t-1}G_{t-1} &= A_tf(x_{t+1}) -A_{t-1}f(x_t) + A_{t-1}L_{t-1}-A_tL_t\\
    & = A_tf(x_{t+1}) -A_{t-1}f(x_t) - a_tf(x_t) - a_t\bigg[ \frac{\Delta f(x_t;\alpha_t(v_t-x_t))}{\alpha_t}
    - \frac{\alpha_t}{2} \cc_f(1+\epsilon)\bigg]\\
    & = A_t(f(x_{t+1}) -f(x_t)) - a_t\bigg[ \frac{\Delta f(x_t;\alpha_t(v_t-x_t))}{\alpha_t} 
    - \frac{\alpha_t}{2} \cc_f(1+\epsilon)\bigg]\\
    & \leq A_t\big(\Delta f(x_t;\alpha_t(v_t-x_t)) + \frac{\alpha_t^2}{2} \cc_f(1+\epsilon)\big)
    - a_t\bigg[ \frac{\Delta f(x_t;\alpha_t(v_t-x_t))}{\alpha_t} - \frac{\alpha_t}{2} \cc_f(1+\epsilon)\bigg].
\end{align*}
To cancel the model term, we choose the step-size so that
\begin{equation*}
    A_t\Delta f(x_t;\alpha_t(v_t-x_t)) = a_t \frac{\Delta f(x_t;\alpha_t(v_t-x_t))}{\alpha_t}\;.
\end{equation*}
This is equivalent to $\alpha_t = a_t/A_t$.
Substituting this identity back into the previous estimate yields
\begin{equation*}
    A_tG_t - A_{t-1}G_{t-1} 
    \leq A_t \frac{a_t^2}{2A_t^2}\cc_f(1+\epsilon) + \frac{a_t^2}{2A_t}\cc_f(1+\epsilon) 
    = \frac{a_t^2}{A_t}\cc_f(1+\epsilon) =: E_t.
\end{equation*}
Let us choose $a_t = 2(t+1)$ which gives 
$A_t = \sum\limits_{i=0}^t a_i = (t+1)(t+2)$ as well as
\begin{equation*}
    \frac{a_t^2}{A_t} = \frac{2^2(t+1)^2}{(t+1)(t+2)}\leq 4.
\end{equation*}
Lemma~\ref{lem:adgt_telescoping} therefore gives
\begin{equation*}
G_t\leq \frac{1}{A_t}\sum\limits_{i=0}^t E_i 
\leq \frac{1}{A_t}\sum\limits_{i=0}^t 4\cc_f(1+\epsilon) 
= \frac{1}{A_t}(t+1)4\cc_f(1+\epsilon) 
= \frac{4\cc_f}{t+2}(1+\epsilon).
\end{equation*}
Hence, $f(x_{t+1})-f(x^*) \leq G_t \leq \frac{4\cc_f}{t+2}(1+\epsilon)$. 
This completes the proof.
\end{proof}
Thus, we obtain a convergence rate for the primal-dual gap that matches the
smooth setting \cite{CGFWSurvey2022,Ja13}.
The intuition behind this is that locally our piecewise linear model is a good approximation for $f$. 
In our setting, Theorem~\ref{thm:approx} substitutes for the L-smoothness assumption used in the smooth case to ensure the boundedness of $\cc_f$.
In applications, one can obtain problem-dependent upper bounds on the constant $\gamma$ in Theorem~\ref{thm:approx}
from the evaluation procedure associated with an abs-smooth representation, see \cite{Gr13};
informally, $\gamma$ plays a role analogous to a gradient-Lipschitz constant in the smooth setting.
We cannot say anything about the convergence of the iterates generated by 
Algorithm~\ref{alg:revised_asfw_hbasfw},
in fact this is not even guaranteed in the smooth case, see \cite{bolte2023iterates}.

We next turn to the heavy ball variant of Algorithm~\ref{alg:revised_asfw_hbasfw}.
The ADGT argument is the same as in Theorem~\ref{thm:p-d}; the only change is that the lower bound is built from the aggregated model
\[
    \Phi_t(v) := \sum\limits_{i=0}^{t} a_i \frac{\Delta f(x_i;\alpha_i(v-x_i))}{\alpha_i},
\]
rather than from the single most recent abs-linearization.
Thus the proof uses the same ADGT template, but we keep a separate corollary because the lower bound $L_t$ now depends on $\Phi_t$.
This also shows that the heavy ball sub-problem may be solved inexactly.

\begin{corollary}\label{cor:pd_hb}(Primal-dual convergence for heavy ball ASFW)
Let $f\in\cc_{\abs}^d(\rr^n)$ be a convex function over a compact convex set $C\subseteq\rr^n$ with diameter $D$. 
For each $t\geq 0$, assume that the heavy ball iterate $v_t$ satisfies
\[
    \Phi_t(v_t) \leq \min_{v\in C}\Phi_t(v) + \eta_t,
    \qquad
    \eta_t := \epsilon a_t\alpha_t \cc_f = \epsilon \frac{a_t^2}{A_t}\cc_f.
\]
Then the iterates $x_t$ of the heavy ball ASFW Algorithm~\ref{alg:revised_asfw_hbasfw} satisfy 
\begin{equation*}
   f(x_{t+1}) -f(x^*)\leq  G_t \leq \frac{4\cc_f}{t+2}(1+\epsilon),
\end{equation*}
where $x^*\in C$ is an optimal solution to problem (\ref{eq:p}),
$\cc_f$ is the curvature constant of $f$, and
$G_t = U_t -L_t$ for $U_t = f(x_{t+1})$ and 
\begin{equation*}
    L_t = \frac{1}{A_t} \bigg( \sum\limits_{i=0}^t a_if(x_i) +\sum\limits_{i=0}^t a_i\bigg[ \frac{\Delta f(x_i;\alpha_i(v_t-x_i))}{\alpha_i} - \frac{ \alpha_i}{2} \cc_f\bigg] - \eta_t \bigg)\;.
\end{equation*}
\end{corollary}

\begin{proof}
Let $a_t \geq 0$ be the weights from Algorithm~\ref{alg:revised_asfw_hbasfw},
$A_t = \sum\limits_{i=0}^t a_i$, and define $\Phi_t$ as above.
By Remark~\ref{rmk:lb},
\begin{equation*}
    f(x^*) \geq f(x_i) + \frac{\Delta f(x_i;\alpha_i(x^*-x_i))}{\alpha_i} - \frac{\alpha_i}{2}\cc_f
\end{equation*}
for each $i=0,\dots,t$.
Multiplying by $a_i$ and summing gives
\begin{align*}
    A_t f(x^*)
    &\geq \sum\limits_{i=0}^t a_if(x_i)
    + \sum\limits_{i=0}^t a_i\bigg[\frac{\Delta f(x_i;\alpha_i(x^*-x_i))}{\alpha_i}
    - \frac{\alpha_i}{2}\cc_f\bigg]\\
    &\geq \sum\limits_{i=0}^t a_if(x_i)
    + \sum\limits_{i=0}^t a_i\bigg[\frac{\Delta f(x_i;\alpha_i(v_t-x_i))}{\alpha_i}
    - \frac{\alpha_i}{2}\cc_f\bigg] - \eta_t
    =: A_tL_t,
\end{align*}
where the second inequality uses the inexact aggregate oracle condition for $\Phi_t(v_t)$.
Hence,
\begin{equation}\label{eq:Lt_hbasfw}
    L_t = \frac{1}{A_t} \bigg( \sum\limits_{i=0}^t a_if(x_i) +\sum\limits_{i=0}^t a_i\bigg[ \frac{\Delta f(x_i;\alpha_i(v_t-x_i))}{\alpha_i} - \frac{ \alpha_i}{2} \cc_f\bigg] - \eta_t \bigg).
\end{equation}
Let $G_t := f(x_{t+1}) - L_t \geq 0$, and use the conventions $A_{-1}=0$, $\Phi_{-1}\equiv 0$, and $\eta_{-1}=0$.
Then
\begin{align*}
    A_tG_t - A_{t-1}G_{t-1}
    &= A_tf(x_{t+1}) - A_{t-1}f(x_t) + A_{t-1}L_{t-1} - A_tL_t\\
    &= A_t(f(x_{t+1})-f(x_t))
    + \Phi_{t-1}(v_{t-1}) - \Phi_t(v_t)
    + \frac{a_t\alpha_t}{2}\cc_f + \eta_t - \eta_{t-1}.
\end{align*}
Since
\[
    \Phi_t(v_t) = \Phi_{t-1}(v_t) + a_t\frac{\Delta f(x_t;\alpha_t(v_t-x_t))}{\alpha_t}
\]
and $v_{t-1}$ is an $\eta_{t-1}$-approximate minimizer of $\Phi_{t-1}$, we have
\[
    \Phi_{t-1}(v_{t-1}) \leq \Phi_{t-1}(v_t) + \eta_{t-1}.
\]
Therefore,
\begin{align*}
    A_tG_t - A_{t-1}G_{t-1}
    &\leq A_t(f(x_{t+1})-f(x_t))
    - a_t\frac{\Delta f(x_t;\alpha_t(v_t-x_t))}{\alpha_t}
    + \frac{a_t\alpha_t}{2}\cc_f + \eta_t.
\end{align*}
By Definition~\ref{def:curvature},
\[
    f(x_{t+1})-f(x_t)
    \leq \Delta f(x_t;\alpha_t(v_t-x_t)) + \frac{\alpha_t^2}{2}\cc_f.
\]
Substituting this estimate and using $\alpha_t = a_t/A_t$ yields
\begin{equation*}
    A_tG_t - A_{t-1}G_{t-1}
    \leq A_t\frac{\alpha_t^2}{2}\cc_f + \frac{a_t\alpha_t}{2}\cc_f + \eta_t
    = \frac{a_t^2}{A_t}\cc_f + \eta_t
    \leq \frac{a_t^2}{A_t}\cc_f(1+\epsilon)
    =: E_t.
\end{equation*}
Choosing $a_t = 2(t+1)$ gives $A_t = (t+1)(t+2)$ and $\frac{a_t^2}{A_t}\leq 4$.
Lemma~\ref{lem:adgt_telescoping} therefore gives
\begin{equation*}
    G_t \leq \frac{1}{A_t}\sum\limits_{i=0}^t E_i
    \leq \frac{1}{A_t}\sum\limits_{i=0}^t 4\cc_f(1+\epsilon)
    = \frac{4\cc_f}{t+2}(1+\epsilon),
\end{equation*}
which proves the claim.
\end{proof}
The heavy ball ASFW method therefore attains the same $\cO(1/t)$ primal-dual rate as the
relaxed vanilla variant presented in Algorithm~\ref{alg:revised_asfw_hbasfw},
and the exact heavy ball method is recovered by setting $\epsilon = 0$. 
In the smooth setting, the heavy ball Frank-Wolfe method has been shown
to admit a tighter primal-dual bound in computations \cite{HB-fw,MaPo25}.
A key distinction from the smooth setting arises in the computation of the
abs-linearization at each iteration: the points $v_t$ are no longer restricted
to being vertices, if the feasible set $C$ is a polyhedron. 
Rather, they may reside at the intersection of a kink and the boundary of the feasible set $C$. 
The value $f(x_{t+1}) - f(x^*)$ cannot be computed in general since we
usually do not know the value of $f(x^*)$.
The primal-dual bound $G_t$, however, is computable and thus can be employed as a stopping criterion, provided we have a reasonable upper bound on $\cc_f$. 

Finally, we record a weaker robustness statement for the case in which the
piecewise linearization $f_{PL,x}(\cdot)$
is convex at every localization point $x$.
The distinction from Theorem~\ref{thm:p-d} and Corollary~\ref{cor:pd_hb} lies
in our formulation of the lower bound $L_t$,
which is given by Lemma~\ref{l:dual_bound_pc_lin}.

\begin{corollary}\label{thm:pd_cvx_pl}(Primal-dual convergence for the piecewise linear model)
    Let $f\in\cc_{\abs}^d(\rr^n)$ with convex piecewise linearization $f_{PL,x}(\cdot)$
    over a compact convex set $C\subseteq\rr^n$ with 
     diameter $D$. For each $t\geq 0$, the iterates $x_t$ of the relaxed ASFW 
     Algorithm~\ref{alg:revised_asfw_hbasfw}
     satisfy 
 \begin{equation*}
   f(x_{t+1}) -f(x^*)\leq  G_t \leq \frac{1}{2}\cc_f + \frac{2 \cc_f}{t+2}(1+2\epsilon),
 \end{equation*}
where $x^*\in C$ is an optimal solution to problem (\ref{eq:p}),
$\cc_f$ the curvature constant of $f$, 
with $G_t = U_t-L_t$ for $U_t = f(x_{t+1})$ and
\begin{equation*}
   L_t =  \frac{1}{A_t} \bigg( \sum\limits_{i=0}^t a_if(x_i) +\sum\limits_{i=0}^t a_i\bigg[ \frac{\Delta f(x_i;\alpha_i(v_i-x_i))}{\alpha_i} - \frac{1}{2}\cc_f - \frac{\alpha_i}{2}\epsilon \cc_f\bigg] \bigg).
\end{equation*}
\end{corollary}

\begin{proof}
    The proof follows similar lines as Theorem~\ref{thm:p-d},
    with a fundamental difference being in our formulation of $L_t$,
    which is obtained from
    Lemma~\ref{l:dual_bound_pc_lin}, i.e.,
    \begin{equation*}
        f(x^*) \geq f(x_t) + \frac{\Delta f(x_t;\alpha_t(v_t-x_t))}{\alpha_t} 
- \frac{1}{2}  \cc_f - \frac{\alpha_t}{2}\epsilon \cc_f.
    \end{equation*}
    Hence
    \begin{align*}
     A_t f(x^*) & \geq \sum\limits_{i=0}^t a_if(x_i) +\sum\limits_{i=0}^t a_i\bigg[ \frac{\Delta f(x_i;\alpha_i(v_i-x_i))}{\alpha_i} - \frac{1}{2}  \cc_f - \frac{\alpha_i}{2} \epsilon \cc_f\bigg] =: A_t L_t.
\end{align*}
From this inequality, one obtains the lower bound $L_t$ as
\begin{equation}\label{eq:L_t_pwlin}
   L_t =  \frac{1}{A_t} \bigg( \sum\limits_{i=0}^t a_if(x_i) +\sum\limits_{i=0}^t a_i\bigg[ \frac{\Delta f(x_i;\alpha_i(v_i-x_i))}{\alpha_i} - \frac{1}{2}  \cc_f - \frac{\alpha_i}{2} \epsilon \cc_f\bigg] \bigg).
\end{equation}
We define the gap $G_t:= f(x_{t+1}) - L_t \geq 0$. Then as before,
    \begin{align*}
    A_tG_t - A_{t-1}G_{t-1} &= A_tf(x_{t+1}) -A_{t-1}f(x_t) + A_{t-1}L_{t-1}-A_tL_t\\
    & = A_tf(x_{t+1}) -A_{t-1}f(x_t) - a_tf(x_t) - a_t\bigg[ \frac{\Delta f(x_t;\alpha_t(v_t-x_t))}{\alpha_t}
    - \frac{1}{2}  \cc_f - \frac{\alpha_t}{2} \epsilon \cc_f\bigg]\\
    & = A_t(f(x_{t+1}) -f(x_t)) - a_t\bigg[ \frac{\Delta f(x_t;\alpha_t(v_t-x_t))}{\alpha_t} 
    - \frac{1}{2}  \cc_f - \frac{\alpha_t}{2} \epsilon \cc_f\bigg]\\
    & \leq A_t\big(\Delta f(x_t;\alpha_t(v_t-x_t)) + \frac{\alpha_t^2}{2} \cc_f(1+\epsilon)\big)
    - a_t\bigg[ \frac{\Delta f(x_t;\alpha_t(v_t-x_t))}{\alpha_t} - \frac{1}{2}  \cc_f - \frac{\alpha_t}{2}
    \epsilon \cc_f\bigg].
\end{align*}
Again we choose $\alpha_t = a_t/A_t$ so that the model term cancels.
Substituting this identity back yields
\begin{equation*}
    A_tG_t - A_{t-1}G_{t-1} 
    \leq A_t \frac{a_t^2}{2A_t^2}\cc_f(1+\epsilon) + \frac{a_t}{2}\cc_f 
    + \frac{a_t^2}{2A_t} \epsilon\cc_f
    = \frac{a_t^2}{2A_t}\cc_f(1+2\epsilon) + \frac{a_t}{2}\cc_f =: E_t.
\end{equation*}
Choosing $a_t = 2(t+1)$ which gives 
$A_t = \sum\limits_{i=0}^t a_i = (t+1)(t+2)$ then results in 
\begin{equation*}
    \frac{a_t^2}{2A_t} = \frac{2^2(t+1)^2}{2(t+1)(t+2)}\leq 2.
\end{equation*}
Lemma~\ref{lem:adgt_telescoping} therefore gives
\begin{align*}
G_t &\leq \frac{1}{A_t}\sum\limits_{i=0}^t E_i 
\leq \frac{1}{A_t}\sum\limits_{i=0}^t \big(2\cc_f(1+2\epsilon) + (i+1)\cc_f \big)\\ 
& = \frac{1}{A_t}(t+1)2\cc_f(1+2\epsilon) + \frac{1}{A_t}  \frac{(t+1)(t+2)}{2}\cc_f
 = \frac{2}{t+2}\cc_f(1+2\epsilon) + \frac{\cc_f}{2},
\end{align*}
which completes the proof.
\end{proof}
\begin{remark}\label{rem:weak_pwlin}
The estimates in Lemma~\ref{l:dual_bound_pc_lin} and Corollary~\ref{thm:pd_cvx_pl} should be interpreted as robustness statements rather than sharp rates.
Since $f_{PL,x}(\cdot)$ need not be a supporting model of $f$, the pointwise upper bound on
$f(x)-f(x^*)$ already incurs the residual term $\cc_f/2$, and the ADGT convergence estimate inherits the same floor.
By Theorem~\ref{thm:approx} and $\cc_f\le 2\gamma D^2$, this floor is at most $\gamma D^2$.
Consequently, if $\cc_f$ is instantiated via a loose upper bound (for example $\cc_f := 2\gamma D^2$)
or if the diameter $D$ is large, then both the pointwise upper bound and the resulting convergence estimate may be quite weak, or even vacuous.
In the convex setting of Theorem~\ref{thm:p-d} and Corollary~\ref{cor:pd_hb}, by contrast, the floor disappears and one recovers a vanishing $\cO(1/t)$ bound.
\end{remark}
 
\section{Numerics}\label{sec:numerics}
For our numerical experiments we use \texttt{AbsSmoothFrankWolfe.jl},
an open-source implementation of the ASFW algorithm in \emph{Julia} \cite{julia}. 
This is part of the \texttt{FrankWolfe.jl} library which provides extensible 
implementations of the Frank-Wolfe algorithms with support for away steps,
line search strategies, and custom constraint oracles, see \cite{frankwolfe}.
The code terminates when the ASFW dual gap $g_t$ (\ref{def:fw_gap}) is 
less than a set tolerance.
It is important to formally distinguish between the theoretical quantity $G_t$ used in our
convergence analysis and the practical stopping criterion $g_t$. 
Furthermore, in the context of the inexact oracles, the computed criterion $\hat{g}_t$ serves 
as a lower bound on the exact criterion $g_t$, i.e., $\hat{g}_t \leq g_t$.
Relying on $\hat{g}_t$ can therefore lead to premature termination of the algorithm.
The number of outer iterations required to achieve the
set tolerance is used as a performance measure but the premature
stopping may underestimate these counts and work in favor of the inexact
approach.  

When the feasible set $C$ is a polyhedron, we employ an adapted
version of the ASM as described above to solve our inner 
sub-problem  using \texttt{HiGHS.jl} as the inner solver, see \cite{KrPoWaWo23} for further details.
We use \texttt{ADOLC.jl} as a \emph{Julia} wrapper of one 
of the most comprehensive
packages for algorithmic differentiation implemented
in C/C++, namely ADOL-C \cite{adolc} to
automatically generate the abs-linearization. 
We now verify our theory on a set of standard non-smooth test problems
from \cite{BaKaMae14}.
We specify the domain of interest $C$ as a 
polyhedron with fixed bounds for the purposes
of numerical experiments.
These bounds are not restrictive and may be adjusted as 
needed for other applications.
\paragraph{The MAXQ function}
The MAXQ function is a convex problem in dimension 20,
with objective function
\begin{equation*}
    f(x) = \max\limits_{1\leq i\leq 20} x_i^2.
\end{equation*}
The starting point is given componentwise by
\begin{equation*}
    \cx_i =
    \begin{cases}
        i, & i=1,\ldots,10,\\
        -i, & i=11,\ldots,20.
    \end{cases}
\end{equation*}
The bounds of the problem are fixed to $-20\leq x_i\leq 20$
for $1\leq i\leq20$, i.e.,
our convex and compact domain $C$ is a polyhedron with diameter $D = 40\sqrt{20}$.
\paragraph{Wong 2 function}
The Wong 2 function is a convex problem in dimension 10, 
with objective function
\begin{equation*}
    f(x) = \max\limits_{1\leq i\leq 9} f_i(x)\;,
\end{equation*}
where
\begin{align*} 
 f_1(x) =\; &  x_1^2+x_2^2+x_1x_2-14x_1-16x_2+(x_3-10)^2+4(x_4-5)^2 +(x_5-3)^2\\
  & +2(x_6-1)^2+5x_7^2+7(x_8-11)^2+2(x_9-10)^2+(x_{10}-7)^2+45, \\ 
 f_2(x) =\; & f_1(x) + 10(3(x_1-2)^2+4(x_2-3)^2+2x_3^2-7x_4-120), \\
 f_3(x) =\; &  f_1(x) + 10(5x_1^2+8x_2+(x_3-6)^2-2x_4-40), \\ 
 f_4(x) =\; &  f_1(x) + 10(0.5(x_1-8)^2+2(x_2-4)^2+3x_5^2-x_6-30),\\ 
 f_5(x) =\; &  f_1(x) + 10(x_1^2+2(x_2-2)^2-2x_1x_2+14x_5-6x_6),\\ 
 f_6(x) =\; &  f_1(x) + 10(4x_1+5x_2-3x_7+9x_8-105),\\ 
 f_7(x) =\; &  f_1(x) + 10(10x_1-8x_2-17x_7+2x_8),\\ 
 f_8(x) =\; &  f_1(x) + 10(-3x_1+6x_2+12(x_9-8)^2-7x_{10}),\\ 
 f_9(x) =\; &  f_1(x) + 10(-8x_1+2x_2+5x_9-2x_{10}-12).\\ 
\end{align*}
The starting point is $\cx = (2,3,5,5,1,2,7,3,6,10)$, and the bounds are $-10\leq x_i\leq 10$ for $1\leq i \leq 10$.
\paragraph{Chained CB3 I function}
The Chained CB3 I function is convex in arbitrary dimension $n$, with objective function
\begin{equation*}
    f(x) = \sum_{i=1}^{n-1} \max \{ x_i^4 + x_{i+1}^2,
(2-x_i)^2 + (2-x_{i+1})^2, 2e^{-x_i + x_{i+1}} \} .
\end{equation*}
The starting point is $\cx = (2,2,\ldots,2)$, and the bounds are $-5\leq x_i\leq 5$ for $1\leq i\leq n$.
\paragraph{Chained Crescent II function}
The Chained Crescent II is a non-convex function of arbitrary dimension $n$, with 
objective function
\begin{equation*}
   f(x) = \sum_{i=1}^{n-1} \max \left\{ x_i^2 + (x_{i+1} - 1)^2 + x_{i+1} - 1, \, -x_i^2 - (x_{i+1} - 1)^2 + x_{i+1} + 1 \right\}.
\end{equation*}
The starting point is given componentwise by
\begin{equation*}
\cx_i = 
\begin{cases} 
  2.0, & i \bmod 2 = 0, \\
  -1.5, & i \bmod 2 = 1 .
\end{cases}
\end{equation*}
We set the bounds $-3\leq x_i\leq 3$ for $1\leq i \leq n$.
\begin{figure}[t]
\begin{center}
   \subfloat[][MAXQ function]{\includegraphics[scale=1]{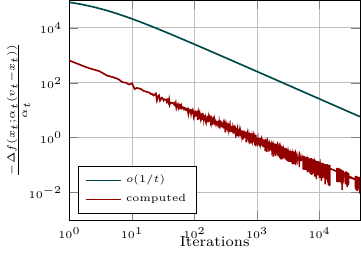}}\hspace{2ex}
\quad
   \subfloat[][Wong 2 function ]{\includegraphics[scale=1]{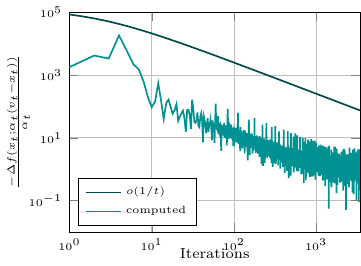}}
\end{center}
\caption{Dual gap $g_t$ versus iterations plotted in log-log scale.}
\label{fig:duals}
\end{figure}

An illustration of the convergence of the dual gaps $g_t$
is presented in Figure~\ref{fig:duals}.
These plots demonstrate that the observed 
convergence rate is consistent with the theoretical
predictions of Theorem~\ref{thm:p-d}.
\begin{figure}[h]
\begin{center}
\subfloat[][MAXQ function]{\includegraphics[scale=1]{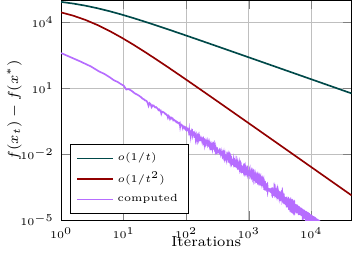}}\hspace{2ex}
\subfloat[][Wong 2 function]{\includegraphics[scale=1]{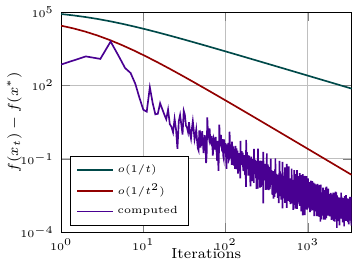}}
\quad
\subfloat[][Chained CB3 I function]{\includegraphics[scale=1]{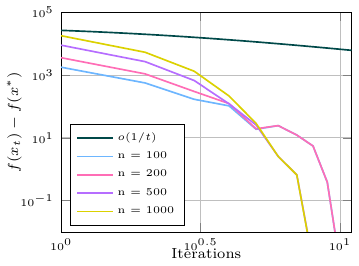}}\hspace{2ex}
\subfloat[][Chained Crescent II function]{\includegraphics[scale=1]{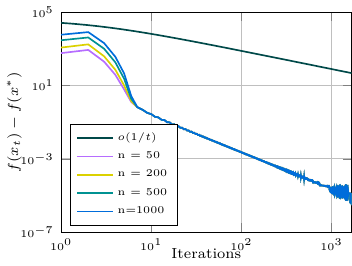}}
\end{center}
\caption{Primal gap $h_t$ against iterations for various functions plotted in log-log scale.}
\label{fig:primals}
\end{figure}

Figure~\ref{fig:primals} illustrates the behavior of the primal gap associated with
Theorem~\ref{thm:p-d} for the 
convex functions MAXQ, 
Wong 2 and 
Chained CB3 I for different $n$.
For the non-convex Chained Crescent II function, we start in a
region where the model is a good local approximation
and as the chosen domain is relatively small,
the residual term is correspondingly small.
As a result, we observe that the primal gap $h_t$
decreases roughly as $\cO(1/t)$, which is qualitatively consistent with the residual-type behavior in
Corollary~\ref{thm:pd_cvx_pl}.
In addition, the primal gap of the MAXQ and Wong 2 test problems exhibits 
accelerated convergence behavior of $\cO(1/t^2)$, 
which is not yet theoretically understood.
In the smooth setting, it is well known that 
under stronger assumptions, such as strong convexity of 
$f$ or uniform convexity of the domain $C$, the rates exhibit an 
accelerated rate of $\cO(1/t^2)$.
Extending our analysis to capture such accelerated behavior,
particularly in the presence of approximate models
is an interesting direction that will be explored in future work.

For the plots in Figure~\ref{fig:duals} and
Figure~\ref{fig:primals}, the inner sub-problem in
Step~\ref{a:s2} of Algorithm~\ref{alg:asfw_hbasfw}
is solved to high accuracy by setting the inner-iteration
limit of the ASM solver to $100$. 
Equivalently, this can be interpreted as 
taking $\epsilon \approx 0$. 

\begin{table}[t]
\centering
\begin{tabular}{lclccr}
\toprule
Function ($n$) 
& Reference $f^*$
& max inner iter (ASM)
& \# iterations
& Relaxed $f(x^*)$
& \# simplex steps \\
\midrule
\multirow{2}{*}{MAXQ (20)}
 & \multirow{2}{*}{00.00}
 & 2
 & 16498
 & 3.348 $\times 10^{-6}$
 & 360546 \\
 & 
 & 10
 & 44010
 & 4.356 $\times 10^{-7}$
 & 976932 \\
 & 
 & 100 \text{(exact)}
 & 44010
 & 4.356 $\times 10^{-7}$
 & 1034172 \\
\multirow{2}{*}{Wong 2 (10)}
 & \multirow{2}{*}{24.3062}
 & 2
 & 2841
 & 24.30652
 & 34093 \\
 & 
 & 10 \text{(exact)}
 & 3411
 & 24.30643
 & 40912 \\
 & 
 & 100 
 & 3411
 & 24.30643
 & 40912 \\
 \multirow{2}{*}{CB3 I (500)}
 & \multirow{2}{*}{998.00}
 & 2
 & 6
 & 998.0000
 & 10479 \\
 & 
 & 10
 & 8
 & 998.0011
 & 50159 \\
 & 
 & 100 \text{(exact)}
 & 8
 & 998.0011
 & 171570 \\
\bottomrule
\end{tabular}
\caption{Reference and relaxed solutions for different inner iteration limits of 
the relaxed ASFW algorithm.}
\label{tab:cvx_examples}
\end{table}
In Table~\ref{tab:cvx_examples} and Table~\ref{tab:noncvx_examples}
we record the maximum inner iterations,
total iterations and solution $f(x^*)$
for all the test functions.
We also count the total number of simplex steps 
used by \texttt{HiGHS.jl} for these functions.
\begin{table}[t]
\centering
\begin{tabular}{lclccr}
\toprule
Function ($n$) 
& Reference $f^*$
& max inner iter (ASM)
& \# iterations
& Relaxed $f(x^*)$
& \# simplex steps \\
\midrule
  \multirow{2}{*}{Crescent II (200)}
 & \multirow{2}{*}{0}
 & 2
 & 1670
 & 3.534 $\times 10^{-6}$
 & 6162 \\
 & 
 & 10
 & 1670
 & 3.534 $\times 10^{-6}$
 & 6297 \\
 &
 & 100 \text{(exact)}
 & 1670
 & 3.534 $\times 10^{-6}$
 & 20635 \\
\bottomrule
\end{tabular}
\caption{Reference versus relaxed solutions for the relaxed ASFW algorithm on the non-convex Chained Crescent II problem.}
\label{tab:noncvx_examples}
\end{table}
Table~\ref{tab:cvx_examples} reports numerical results for
the convex benchmark problems - MAXQ, Wong 2 and Chained CB3 I,
using the relaxed ASFW algorithm with
different limits on the number of inner iterations of 
the ASM method.
In \cite{GW18} it was shown that ASM converges to a local minima of a piecewise linear function in a
finite number of iterations. 
In Table \ref{tab:cvx_examples}, the 'max inner iter' setting effectively yields
an exact method because no ASFW iteration ever triggers this maximum limit.
The results indicate that even a small number of 
inner iterations (e.g., 2) is sufficient to recover
highly accurate objective values, closely matching
the reference optima.
Increasing the number of inner iterations beyond this
threshold does not lead to noticeable improvements 
in solution quality, but substantially increases the
computational cost, as reflected by the higher number
of simplex steps.
This demonstrates that relatively loose inner solves
can already provide excellent practical performance on
convex instances.

Table~\ref{tab:noncvx_examples} presents verified results for the 
non-convex Chained Crescent II problem with dimension $n=200$. 
Interestingly, the algorithm converges to identical 
objective values for all the reported inner iteration limits, 
suggesting that the outer Abs-Smooth Frank-Wolfe scheme is
robust to inexact inner solves even in the non-convex regime.
However, tightening the inner accuracy leads to a dramatic
increase in computational cost, with the number of simplex 
steps increasing by roughly a factor of three.
This highlights a clear trade-off between inner solve accuracy
and overall efficiency
and supports the use of relaxed inner solutions in practice.

We now compare ASFW to non-smooth solvers in \emph{Julia}. 
In fact, relatively few optimization algorithms in Julia are designed 
to handle non-smooth objectives efficiently.
While it is possible to compare \texttt{AbsSmoothFrankWolfe.jl} with
the gradient free Nelder-Mead algorithm \cite{nelder1965simplex} available 
with the optimization framework \texttt{Optim.jl} \cite{mogensen2020optim},
such a comparison is not entirely meaningful,
as Nelder-Mead is a general purpose derivative free optimizer
that does not exploit the structure of non-smooth functions.
Similarly, methods based on \texttt{ProximalAlgorithms.jl} \cite{proximal2, proximal1}
require identifying a smooth component of the objective and
computing its Lipschitz constant, which is nontrivial or
even impractical for functions such as Wong 2 and CB3 I. 
Therefore, the most natural and informative comparison for ASFW
is with the subgradient based Frank-Wolfe method,
which, like ASFW, is specifically designed to handle non-smooth objectives
without relying on Lipschitz constants or explicit gradients.
We thus focus our numerical evaluation on this method,
as it provides a meaningful and representative point of reference
for assessing the performance of ASFW.

We use the same starting point and feasible domain $C$ for these methods.
In view of the empirical observations reported above, we adopt a maximum
inner iteration limit of 2 throughout the remainder of the numerical experiments
and present results obtained under this configuration.
\begin{table}[h]
\centering
\begin{tabular}{lrcccc}
\toprule
Function ($n$) & Reference $f^*$
& \multicolumn{2}{c}{\texttt{AbsSmoothFrankWolfe.jl}}
& \multicolumn{2}{c}{\texttt{FrankWolfe.jl} with subgradients} \\
\cmidrule(lr){3-4} \cmidrule(lr){5-6}
& & \# iter & $f(x^*)$ 
& \# iter & $f(x^*)$  \\
\midrule
MAXQ (20) & 00.00  & 16498 & 3.348 $\times 10^{-6}$  
                & 20001* & 44.3329  \\
Wong 2 (10) & 24.3062  & 2841 & 24.30652
                & 10001* & 114.7234  \\
CB3 I (300) & 598.00  & 6 & 598.0000  
                & 10001* & 598.9557  \\
Crescent II (1000) & 00.00  & 1670 &  3.534 $\times 10^{-6}$
                          & 10001*  & 0.59904 \\
\bottomrule
\end{tabular}
\caption{Performance comparison between ASFW and Frank-Wolfe with subgradients.}
\label{tab:fw_compare}
\end{table}
Table~\ref{tab:fw_compare} reports a comparison between \texttt{AbsSmoothFrankWolfe.jl}
and the subgradient based \texttt{FrankWolfe.jl} method on four non-smooth benchmark problems.
For each instance, we present the number of iterations and the final objective value,
together with a reference optimal value.
The subgradient based method fails to satisfy the stopping criterion within
the prescribed iteration budget and is therefore terminated at the maximum number
of iterations, as indicated by an asterisk (*).
In contrast, ASFW consistently converges to near-optimal objective values 
within a finite number of iterations across all test cases.
It is worth noting that one FW iteration requires solving a 
single linear minimization oracle, whereas ASFW often requires solving a more complex piecewise linear sub-problem.
On  MAXQ and Wong~2, the subgradient method also ends with
substantially worse objective values.
On CB3~I and Crescent~II, by contrast, the final objective values are competitive,
but the subgradient method still fails the stopping criterion within the iteration budget,
showing that objective values alone can mask a large residual Frank-Wolfe gap on non-smooth problems.
This illustrates the practical advantages of the proposed abs-smooth 
framework for non-smooth optimization.

\subsection{The LASSO problem}\label{p:lasso}
We now demonstrated the efficiency of our algorithm on the well-known convex LASSO problem.
The LASSO is the least-squares linear regression objective with an additional $\ell_1$ regularization term. 
The LASSO objective is
\begin{equation*}
   f(x) = \frac{1}{2}\|Ax - y\|^2 + \rho \|x\|_1 \;,
\end{equation*}
with $y\in \rr^p$ being a response vector, 
$A\in \rr^{p\times n}$ being a design matrix. 
The LASSO problem is the sum of a smooth convex function and a non-smooth convex function.

We evaluate the performance of the LASSO model using a real-world dataset,
specifically the diabetes dataset, in which the model is applied to predict 
the progression of diabetes based on several clinical and physiological predictors.
The diabetes dataset is a widely used benchmark dataset introduced 
in \cite{lasso_diabetes}. 
It consists of clinical and physiological measurements from patients, 
along with a continuous response variable representing disease progression 
one year after baseline. The dataset is commonly used to evaluate regression 
and variable selection methods, including regularization techniques such as LASSO.
For the diabetes benchmark in Table~\ref{tab:lasso_diabetes}, we accessed the dataset
through OpenML \cite{OpenML44223}, 
which provides a convenient,
standardized format for machine learning tasks.

\begin{table}[h]
\centering
\begin{tabular}{lcccr}
\toprule
$\rho$ & \# iter & intercept & MSE & \# simplex \\
\midrule
0.1  & 6132  & 152.13348 & 2859.89163 & 63658 \\
0.5  & 5914  & 152.13348 & 2859.92992 & 61494 \\
1    & 5611  & 152.13348 & 2859.97013 & 58529 \\
5    &  6672 & 152.13348 & 2860.10035 & 69521 \\
10    & 6517  & 152.13348 & 2860.42424 & 69944 \\
\bottomrule
\end{tabular}
\caption{MSE and intercept for various $\rho$ on the diabetes dataset.}
\label{tab:lasso_diabetes}
\end{table}
For the diabetes dataset, the design matrix $A \in \rr^{442\times 10}$
consists of ten standardized clinical and physiological predictors, while the response vector 
$y \in \rr^{442}$ represents the quantitative measure of disease progression one year after baseline.
An intercept term captures the baseline disease progression and 
is expected to be close to the mean response value, 
approximately $152.13$ for this dataset.
Model performance is assessed using the mean squared error (MSE), 
which is typically in the range
$2859-2879$ after training.
We fix the feasible set as $C = [-50,50]$ after an appropriate scaling of the data and report
the intercept, MSE and the total number of simplex steps used by the ASFW
for different values of $\rho$ in Table~\ref{tab:lasso_diabetes}.

\section{Summary and outlook}
In this work, we provided a comprehensive primal-dual analysis for convex abs-smooth functions using
the ASFW framework.  
We introduced the relaxed ASFW algorithm and the momentum based variants of the ASFW algorithm, 
a robust framework for solving non-smooth optimization problems.
The proposed methods are specifically tailored to address
the challenges posed by non-differentiability and 
the absence of smooth structure, which commonly arise in practical applications.
Through rigorous theoretical analysis, we established that
the relaxed vanilla and heavy ball variants admit matching $\cO(1/t)$ bounds
on the approximate duality gap, analogous to the classical Frank-Wolfe rate for smooth
convex objectives. 
Furthermore, when only the piecewise linearizations are convex,
the ADGT analysis still yields a residual-type estimate of the form
$G_t \le \cc_f/2 + \cO(1/t)$.
This bound can be weak when the model mismatch is large, but it nevertheless
provides a robustness guarantee beyond the fully convex setting.
Notably, this performance is attained without
relying on smoothing techniques or requiring access to
generalized gradient information, thereby preserving the original
problem structure and enhancing the applicability of
ASFW to a broad class of non-smooth optimization problems.

Our study raises several important questions that warrant 
further investigation and form the basis for future research.
Future work will therefore focus on addressing these open 
questions, extending the theoretical foundations of ASFW, 
and applying the algorithm to large-scale machine learning
problems. 
In particular, we aim to study accelerated convergence regimes
and explore different agnostic step-size strategies.
Finally, we plan to investigate algorithmic variants 
such as the away-step ASFW and stochastic ASFW, which are well suited for 
high-dimensional and data-driven applications.

\section*{Acknowledgments}
This project is funded by the Deutsche Forschungsgemeinschaft (DFG, German Research
Foundation) under Germany's Excellence Strategy - The Berlin Mathematics
Research Center MATH+ (EXC-2046/1, EXC-2046/2, project ID: 390685689).

\printbibliography

\end{document}